\documentclass[]{amsart}

\usepackage{amsmath, amsthm, amscd, amsfonts, amssymb, graphicx, color}
\usepackage{color}
\usepackage{url}
\usepackage{tikz-cd}
\usepackage{xcolor}

\newtheorem{theorem}{Theorem}[section]
\newtheorem{lemma}[theorem]{Lemma}
\newtheorem{corollary}[theorem]{Corollary}
\newtheorem{proposition}[theorem]{Proposition}

\newtheorem{definition}[theorem]{Definition}
\newtheorem{example}[theorem]{Example}
\newtheorem{problem}{Problem}[section]

\newtheorem{remark}{Remark}

\begin{document}
	
	\title[]{Finding singularly cospectral graphs}
	
	\author[C.M. Conde]{Cristian M. Conde${}^{1,3}$}
	\author[E. Dratman]{Ezequiel Dratman${}^{1,2}$}
	\author[L.N. Grippo]{Luciano N. Grippo${}^{1,2}$}
	
	\address{${}^{1}$Instituto de Ciencias\\Universidad Nacional de General Sarmiento}
	\address{${}^{2}$Consejo Nacional de Investigaciones Cient\'ificas y Tecnicas, Argentina}
	\address{${}^{3}$Instituto Argentina de Matem\'atica "Alberto Calder\'on" -  Consejo Nacional de Investigaciones Cient\'ificas y Tecnicas, Argentina}
	
	\email{cconde@campus.ungs.edu.ar}
	\email{edratman@campus.ungs.edu.ar}
	\email{lgrippo@campus.ungs.edu.ar}

	
	\date{}
	
	\maketitle

\begin{abstract}
	Two graphs having the same spectrum are said to be cospectral. A pair of singularly cospectral graphs is formed by two graphs such that the absolute values of their nonzero eigenvalues coincide. Clearly, a pair of cospectral graphs is also singularly cospectral but the converse may not be true. Two graphs are almost cospectral if their nonzero eigenvalues and their multiplicities coincide. In this paper, we present necessary and sufficient conditions for a pair of graphs to be singularly cospectral, giving an answer to a problem posted by Nikiforov. In addition, we construct an infinite family of pairs of noncospectral singularly cospectral graphs with unbounded number of vertices.  It is clear that almost cospectral graphs are also singularly cospectral but the converse is not necessarily true, we present families of graphs where both concepts: almost cospectrality and singularly cospectrality agree.
\end{abstract}

\section{Introduction}\label{s1}

Two graphs are said to be cospectral if they have the same spectrum. The problem of finding families of pairs of nonisomorphic cospectral graphs have attracted the attention of many researchers. Probably, the first relevant result on this subject is due to Schwenk, who proves that almost all trees are cospectral~\cite{schwenk1973}. Since then, many articles, presenting constructions to either generate pairs of cospectral graphs or finding families of graphs which do not have a mate, known as defined by their spectrum, have been published. See for instance~\cite{GodsilM82,vanDamH09,Quietal20}. The energy of a graph was defined by Gutman in 1978 as the sum of the absolute values of its eigenvalues, counted with their multiplicity~\cite{Gutman78}. Two graphs, with the same number of vertices, are said to be equienergetic if they have the same energy. Clearly, two cospectral graphs are also equienergetic. Nevertheless, there are examples of pairs of noncospectral equienergetic graphs~\cite{Stevanovic05,BonifacioVA08}. Indeed, finding noncospectral equienergetic pairs of graphs is a very active and actual topic of research. The energy of a graph is nothing but the trace norm of its adjacency matrix; i.e., the sum of the singular values of its adjacency matrix. Notice that, when a matrix is symmetric, its singular values are precisely the absolute values of its eigenvalues. In~\cite{Nikiforov2016}, Nikiforov defines two graphs as singularly cospectral if their nonzero singular values, counted with their multiplicity, coincide. Hence any two singularly cospectral graphs are equienergetic. In that article, he also posted the following problem.

\begin{problem}~\cite{Nikiforov2016}\label{Niki}
	Find necessary and sufficient conditions for two graphs to be singularly cospectral.
\end{problem}

We give a characterization in Section~\ref{sec: main result} that answers the Problem~\ref{Niki}. In addition, we formulate the following two natural problems, in connection with the notion of singularly cospectral.

\begin{problem}~\label{problem 1}
	Find infinite pairs of noncospectral singularly cospectral graphs.
\end{problem}


\begin{problem}~\label{problem 2}
	Find families of graphs where singularly cospectrality implies cospecrality.
\end{problem}

In Section~\ref{sec: construction sc noncospectral} we deal with Problem~\ref{problem 1} by presenting a construction based on the spectral decomposition of a symmetric matrix. Finally, in Section~\ref{sec: c implies sc}, we define families of graphs by imposing certain constrains on the spectrum of their graphs that makes equivalent the notions of cospectrality and singularly cospectrality, answering the Problem~\ref{problem 2}.  


\section{Preliminaries}\label{s2}

All graphs, mentioned in this article are finite, have no loops and multiple edges. Let $G$ be a graph. We use $V(G)$ and $E(G)$ to denote the set of $n=|V(G)|$ vertices and the set of $m=|E(G)|$ edges of $G$, respectively. 
The adjacency matrix associated with the graph $G$ is defined by
$A_G= (a_{ij})_{n\times n}$, where $a_{ij} = 1$; if $v_i \sim  v_j$  i.e., the vertex $v_i$ is adjacent to the vertex
$v_j$ and $a_{ij} = 0$ otherwise.  We often denote $A_G$ simply by $A.$ We use $d_G(v)$ to denote the degree of $v$ in $G$ (the number of edges incident to $v$), or $d(v)$ provided the context is clear.  A $d$-regular graph is a graph such that $d_G(v)=d$ for every $v\in V(G)$.

For a matrix $M \in \mathbb{R}^{n\times n}$, the
spectrum,  $\sigma(M)$ is the
multiset $
\sigma(M)=\{[\lambda_1(M)]^{m_1},\cdots, [\lambda_{r}(M)]^{m_{r}}\},
$ where $\lambda_1(M)>\lambda_2(M)>\cdots >\lambda_r(M)$ are the distinct eigenvalues of $M$ and $m_i$
is the multiplicity
of $\lambda_i(M)$ for any $i \in \{1,\cdots, r\}$ and spectral radius, $\rho(M)$  is the
largest absolute value of the eigenvalues i.e., largest of $|\lambda_1(M)|, |\lambda_2(M)|, \cdots, |\lambda_r(M)|$ where $\lambda_i(M)\in \sigma(M)$. 
To denote the spectrum of $A_G$ we use $\sigma(G)$. 


Next proposition collects several results about the eigenvalues of a graph. Most of these statements are well-known.

\begin{proposition}{\cite[Chapter 6]{Ba}}\label{propiedades_basicas}
	Let $G$ be a  graph with $n$ vertices and  $m$ edges. Then, 
	\begin{enumerate}
		\item $|\lambda_i(G)|\leq \lambda_1(G)$ for all $i\geq 2.$
		\item  $\lambda_1(G) \geq \frac{2m}{n}.$ Equality holds if and only $G$ is a $\frac{2m}{n}$-regular.
		\item If $G$ is connected, the equality $\lambda_r(G)=-\lambda_1(G)$ holds if and only $G$ is bipartite. 
		\item If $G$ is bipartite, then $\sigma(G)$ is symmetric respect to zero.
	\end{enumerate}
\end{proposition}

The inertia of a graph $G$ is the triple $In(G) = (p(G), z(G), n(G)),$ in which $p(G), z(G), n(G)$ stand for the number of positive, zero,
and negative eigenvalues of $G$, respectively. The energy of $G$,
denoted by $\mathcal{E} = \mathcal{E}(G)$, is the sum of the absolute values of the eigenvalues of $G$. 

Let $G$ and $H$ be two graphs with vertex sets $V(G)$ and $V(H)$, respectively. The \emph{strong product of $G$ and $H$} is the graph $G\boxtimes H$ such that the vertex set is $V(G)\times V(H)$ and in which two vertices $(g_1, h_1)$ and $(g_2, h_2)$ are adjacent if and only if one of the following conditions holds: $g_1=g_2$ and $h_1$ is adjacent to $h_2$, $h_1=h_2$ and $g_1$ is adjacent to $g_2$, or $g_1$ is adjacent to $g_2$ and $h_1$ is adjacent to $h_2$. The \emph{cartesian product of $G$ and $H$} is the graph $G\square H$ such that the vertex set is $V(G)\times V(H)$ and in which two vertices $(g_1, h_1)$ and $(g_2, h_2)$ are adjacent if and only if either $g_1 = g_2$ and $h_1$ is adjacent to $h_2$ in $H$ or $g_1$ is adjacent to $g_2$ in $G$ and $h_1 = h_2.$

We denote by $C_n$, $P_n$ and $K_n$ to the cycle, the path and the complete graph on $n$ vertices. By $K_{r,s}$ we denote the complete bipartite graphs with a bipartition on $r$ and $s$ vertices respectively.  We use $P_G(x)$ to denote the characteristic polynomial of $A_G$; i.e., $P_G(x)=\det(xI_n-A_G)$. It is easy to prove that $P_{K_n}(x)=(x-n+1)(x+1)^{n-1}$ and in consequence $\sigma(K_n)=\{[n-1]^1,[-1]^{n-1}\}$.

\subsection*{Schatten norms}


Consider a matrix $M \in \mathbb{R}^{n\times n}$. 
Recall that the {\textit singular values} of a matrix $M$ are the square roots of the eigenvalues of $M^*M$, where $M^*$ is the conjugate transpose of $M$. We denote by  $s_1(M), s_2(M),\cdots, s_r(M)$ for the singular values of $M$ arranged in descending order.  Let $r={\rm rank}(M)$ then  $s_{r+1}(M) =\cdots=s_n(M) = 0.$

Define for $p>0$, 
$$
\|M\|_p=\left(\sum_{i=1}^n s_i(M)^p \right)^{\frac 1p}.
$$
For $p\geq 1$, it is  a norm over ${\mathbb R}^{n\times n}$ called the {\textit{ $p$-Schatten norm}}. When $p=1$, it is also called the trace norm or nuclear norm. When $p= 2$, it is exactly the Frobenius norm $\|M\|_2.$

\begin{remark}
	As $M^*M$ is a positive semi-definite matrix, so the eigenvalues of  $(M^*M)^{1/2}$ coincide with the  singular values of $M$. It holds that 
	\begin{equation}\label{limitenormap}
	\lim\limits_{p\to \infty} \|M\|_p=s_1(M).\end{equation}
\end{remark}

\section{Characterizing singularly cospectral graphs}~\label{sec: main result}

Given a graph $G$, following Nikiforov's notation~\cite{Nikiforov2016}, consider  the function $f_G (p)$ for any $p\geq 1$ as:
$$
f_G(p):=\|G\|_p=\left(\sum_{i=1}^n s_i(G)^p \right)^{\frac 1p},
$$
where $\|G\|_p$ and $s_i(G)$ stands for $\|A_G\|_p$ and $s_i(A_G)$, respectively. The following statement  collects some of the most important known properties of $f_G$.

\begin{lemma}\label{propbasicas}~\cite{Nikiforov2016} \label{Nikiforov}
	The following statements hold for a graph $G$.
	\begin{enumerate}
		\item $f_G(p)$ is differentiable in $p$.
		\item  $f_G(p)$ is decreasing in $p.$
		\item  
		If $G$ is a graph with $m$ edges, then $f_G(2)=\sqrt{2m}.$ Furthermore, for any $k > 1$, the number of closed walks of length $2k$ of a graph $G$ is equal to 
		$$\frac{\left(f_G (2k)\right)^{2k}}{4k}.$$
	\end{enumerate}
	
\end{lemma}

Following \cite{Nikiforov2016} we recall the next definition. 

\begin{definition}Two graphs $G$ and $H$ are called {\bf singularly cospectral} if they have the same nonzero
	singular values with the same multiplicities. We denote a pair of noncospectral singularly cospectral graphs by \textbf{NCSC}.
\end{definition}

We use ${\rm rank}(G)$ and ${\rm nullity}(G)$ to denote ${\rm rank}(A_G)$ and ${\rm nullity}(A_G)$ for a graph $G$, respectively. Let us start stating some  relevant  properties satisfied by singularly cospectral graphs.

\begin{lemma}\label{prop_nec}
	Let $G$ and $H$ be singularly cospectral graphs. Then, the following condition holds. 
	\begin{enumerate}
		\item $G$ and $H$ have the same number of edges.
		\item ${\rm rank}(G)={\rm rank}(H),$ $\big |{\rm nullity}(G)-{\rm nullity}(H)\big|=\big ||V(G)|-|V(H)|\big |$ and $p(G)-p(H)=n(H)-n(G)$.		
		\item $\mathcal{E}(G)=\mathcal{E}(H)$, i.e, $G$ and $H$ are equienergetic.
		
	\end{enumerate}
\end{lemma}
\begin{proof}
	We denoted by $m_G$ and $m_H$ the number of edges of $G$ and $H$, respectively.  By item (4) in Lemma \ref{propbasicas} we have that $\sqrt{2m_G}=f_G(2)=f_H(2)=\sqrt{2m_H}.$ Then, $m_G=m_H.$
	
	Recall that if  $A \in \mathbb{R}^{n\times n}$,  the number of nonzero singular values of $A$ is  equal to  the
	${\rm rank}(A)$. Then ${\rm rank}(G)={\rm rank}(H)$. The rank of any square matrix equals the number of nonzero eigenvalues (with repetitions), so the number of nonzero singular values of $A$ equals the rank of $A^TA$ (notice that $A^T=A^*$). As $A^TA$ and $A$ have the same kernel then it follows from the Rank-Nullity Theorem that $A^TA$ and $A$ have the same rank. On the other hand, from the Rank-Nullity Theorem we have 
	\begin{eqnarray}
	\big |{\rm nullity}(G)-{\rm nullity}(H)\big|&=&\big |\left(|V(G)|-{\rm rank}(G)\right)-\left(|V(H)|-{\rm rank}(H)\right)\big |\nonumber\\
	&=&\big ||V(G)|-|V(H)|\big |.\nonumber\
	\end{eqnarray}
	Finally, since  $p(G)+n(G)={\rm rank}(G)={\rm rank}(H)=p(H)+n(H)$ then $p(G)-p(H)=n(H)-n(G).$
	
	It follows from the fact that the singular values of a graph $G$ coincide with the absolute values of the eigenvalues of $G$.
\end{proof}

\begin{example}{\bf Coenergetic does not imply Singularly cospectral  }
	
	We present a pair of coenergetic nonsingularly cospectral graphs. For this
	purpose, we consider the cartesian product and the strong  product of two complete graphs. Let $r,s$ two integer number with $r>3$ and $s>3$. Let $G=K_r \boxtimes K_s$ and $H=K_r\square K_s$. Then $\mathcal{E}(G)=\mathcal{E}(H)$; i.e., $G$   and $H$ are coenergetic~\cite{BonifacioVA08} but we have that such graphs are no singularly cospectral since
	$|E(G)|=\frac{rs}{2}(s+r-2)$ and $|E(H)|=\frac{rs}{2}(r-1)(s-1).$ 
\end{example}

In the sequel, the following lemma is useful to obtain a characterization for a pair of singularly cospectral graphs. 

\begin{lemma}\label{normakn}
	Let $x_1, x_2,\cdots, x_r$ and $y_1, y_2,\cdots, y_s$ be non-increasing sequences of
	positive real numbers. If
	\begin{equation}\label{igualdadkn}
	\sum_{i=1}^r x_i^{k_n}=\sum_{j=1}^s y_j^{k_n} \qquad {\rm for\: all}\:  n\in\mathbb{N},
	\end{equation}
	where $\{k_n\}_{n\in \mathbb{N}}$ is a sequence of real numbers such that $k_n\geq 1$ for all $n\in \mathbb{N}$ and $k_n\to \infty$, 
	then $r=s$ and $x_i=y_i$ for $i=1, \cdots, r.$
\end{lemma}
\begin{proof}
	We denote by $x=(x_1, x_2,\cdots, x_r, 0, 0, \cdots)$ and $y=(y_1, y_2,\cdots, y_s, 0, 0, \cdots)$.  We have, by  \eqref{igualdadkn}, that $\|x\|_{k_n}=\|y\|_{k_n}$ for all $k_n$, then from \eqref{limitenormap} we have that 
	$$x_1=\|x\|_{\infty}=\lim\limits_{n\to \infty} \|x\|_{k_n}=\lim\limits_{n\to \infty} \|y\|_{k_n}=\|y\|_{\infty}=y_1.$$
	By repeating this procedure we can prove that 
	\begin{equation}\label{igualdadi}
	x_i=y_i \qquad \forall  i=1, 2, \cdots, \min\{r, s\}.
	\end{equation}
	Now we assume that $r<s.$ Then by \eqref{igualdadkn} and \eqref{igualdadi} we obtain that $\sum_{j=r+1}^s y_j^{k_n}=0$, as $y_j$ is a positive real  number for any $j=r+1, \cdots, s$ we  conclude that $y_j=0$ for all $j\geq r+1.$ This shows that $r=s$ and $x_i=y_i$ for all $i=1, \cdots, r.$
\end{proof}

We can then prove the following theorem for characterizing when two graphs $G$ and $H$ are singularly cospectral graphs in terms of their functions $f_G(x)$ and $f_H(x)$.

\begin{theorem} \label{equivalencia}
	The following conditions are equivalent:
	\begin{enumerate}
		\item $G$ and $H$ are singularly cospectral.
		
		\item  $f_G(p)=f_H(p)$ for all $p\geq 1.$
		\item $f_G(x_n)=f_H(x_n)$ for any sequence $\{x_n\}_{n\in \mathbb{N}}$ such that $x_n\geq 1$ for all $n\in \mathbb{N}$ and $x_n\to \infty$.
	\end{enumerate}
\end{theorem}

\begin{proof}
	The implications $1\Rightarrow 2$ and $2\Rightarrow 3$ are trivial.  Now,  suppose that 3 holds. Then
	$$
	\sum_{i=1}^r s_i(G)^{x_n}=\sum_{j=1}^s s_j(H)^{x_n} \qquad {\rm for\: all}\:  n\in\mathbb{N},$$
	hence by Lemma \ref{normakn} we have that $G$ and $H$ are singularly cospectral.
	
\end{proof}

Now, we are ready to prove the main result of this section, giving an answer to Problem~\ref{Niki}.

\begin{theorem}\label{singcospectral}
	Two graphs are singularly cospectral if and only if, for each $k\in \mathbb{N}$,
	they have the same number of closed walks of length $2k.$
\end{theorem}


\begin{proof}
	Assume that $G$ and $H$ are singularly cospectral, then $s_i(G)=s_i(H)$ for all $i=1, \cdots, {\rm rank}(G)={\rm rank}(H)$ and $f_G(2k)=f_H(2k)$ for all $k\in \mathbb{N}.$ This implies, by  Lemma \ref{propbasicas},  that the number of closed walks of length $2k$ of $G$ and $H$ coincides.
	
	Conversely, assume that $G$ and $H$ have the same number of closed walks of length $2k$ for all $k\in \mathbb{N}.$ This implies, by Lemma \ref{Nikiforov}, that $f_G(2k)=f_H(2k)$ for any $k\in \mathbb{N}$. Then by Theorem \ref{equivalencia} we conclude that $G$ and $H$ are singularly cospectral.
\end{proof}


The following example  is an immediate consequence of previous Theorem. We exhibit a pair of singularly cospectral graphs. Let $G=C_{2j}$ and $F$ be the union disjoint of two copies of the cycle $C_j$  with $j\ge 3$. We use $V(F)=\{v_1, v_2, \cdots, v_{2j}\}=V(G)$ to denote the set of vertices of $F$ and $G$. For each $k\in \mathbb{N}$, we denote by $W^{G}_{2k}=\{v_{n_i}\}_{i=1,\cdots, 2k+1} $  a closed walk of lenght $2k$ with its endpoint equal to  $v_{n_1}=v_{n_{2k+1}}$  in $G$. We define the following function between the closed walks of even length in $G$ ($\mathcal{F}_G$) and the closed walks of even length in $F$ ($\mathcal{F}_F$) as follows:
$$
\phi: \mathcal{F}_G\to  \mathcal{F}_F, \phi(W^{G}_{2k})=W^{F}_{2k}=\{v_{m_i}\}_{i=1,\cdots, 2k+1}, 
$$
where $m_i = r_j(n_i)$ if $n_1\in \{1, 2, \cdots, j\}$ or $m_i = r_j(n_i) + j$  if $n_1\in \{j+1,\ldots, 2j\}$ for each $1\le i\le 2k+1$, notice that $r_j(n_i)$ represent the remainder when integer $n_i$ is divided by $j$. 
Clearly this is a bijection between $\mathcal{F}_G$ and $\mathcal{F}_F$ and by Theorem \ref{singcospectral} we conclude that this pair of graphs are singularly cospectral.

Even though Theorem~\ref{singcospectral} gives a characterization that may be useful from a theoretic perspective, it does not seem to be helpful from a practical point of view when we try to construct pairs of singularly cospectral graphs, unless these graphs are well-structured as in the last example.

\section{Constructing pairs of \textbf{NCSC}}~\label{sec: construction sc noncospectral}

Recall that two graphs are called cospectral if they have the same spectrum. Obviously, cospectral graphs are  singularly cospectral, but the converse may not be true. See Fig.~\ref{fig: SCNC uno conexo y el otro disconexo} for an example.

As the adjacency matix of a graph is a symmetric matrix with real entries, then it has a spectral decomposition $A=\sum_{i=1}^m \mu_i P_i$ where $\mu_1, \mu_2, \cdots, \mu_m$ are the distinct
eigenvalues of $G$ and $P_i$ represents the orthogonal projection onto the eigenspace $E_{\mu_i}$. In the next result we present how to construct a pair of singularly cospectral graphs but not cospectral from a graph using its spectral decomposition.

\begin{figure}
	\centering
	\begin{tikzpicture}[scale=0.75]
	
	\coordinate (y1) at (7.5,2);
	\coordinate (y2) at (8.75,1); 
	\coordinate (y3) at (6.25,1); 
	\coordinate (y4) at (7.5,-2);
	\coordinate (y5) at (7.5,0);
	
	\coordinate (y6) at (10.75,2);
	\coordinate (y7) at (12,1); 
	\coordinate (y8) at (9.5,1); 
	\coordinate (y9) at (10.75,-2);
	\coordinate (y10) at (10.75,0);
	\coordinate (y11) at (9.25, -3);
	
	\foreach \punto in {1,...,10}
	\fill (y\punto) circle(3pt);
	
	\node [above] at (y1) {$1$};
	\node [above] at (y2) {$2$};
	\node [above] at (y3) {$3$};
	\node [below] at (y4) {$4$};
	\node [below] at (y5) {$5$};
	\node [above] at (y6) {$6$};
	\node [above] at (y7) {$7$};
	\node [above] at (y8) {$8$};
	\node [below] at (y9) {$9$};
	\node [below] at (y10) {$10$};
	
	\node[below] at (y11) {$H_F=2F$};
	
	\draw[thick] (y1)--(y2);
	\draw[thick] (y3)--(y2);
	\draw[thick] (y1)--(y3);
	\draw[thick] (y2)--(y4);
	\draw[thick] (y3)--(y4);
	\draw[thick] (y3)--(y5);
	\draw[thick] (y2)--(y5);
	\draw[thick] (y7)--(y6);
	\draw[thick] (y7)--(y8);
	\draw[thick] (y6)--(y8);
	\draw[thick] (y10)--(y8);
	\draw[thick] (y7)--(y10);
	\draw[thick] (y9)--(y8);
	\draw[thick] (y7)--(y9);
	
	\coordinate (z1) at (2.5,2);
	\coordinate (z2) at (2.5,1); 
	\coordinate (z3) at (2.5,0); 
	\coordinate (z4) at (2.5,-1);
	\coordinate (z5) at (2.5,-2);
	\coordinate (z6) at (4, 2);
	\coordinate (z7) at (4,1);
	\coordinate (z8) at (4, 0); 
	\coordinate (z9) at (4,-1); 
	\coordinate (z10) at (4,-2);
	\coordinate (z11) at (3.25,-3);

	\foreach \punto in {1,...,10}
	\fill (z\punto) circle(3pt);
	
	\node [left] at (z1) {$1$};
	\node [left] at (z2) {$2$};
	\node [left] at (z3) {$3$};
	\node [left] at (z4) {$4$};
	\node [left] at (z5) {$5$};
	\node [right] at (z6) {$6$};
	\node [right] at (z7) {$7$};
	\node [right] at (z8) {$8$};
	\node [right] at (z9) {$9$};
	\node [right] at (z10) {$10$};
	
	\node[below] at (z11) {$G_F=F\times K_2$};
	
	\draw[thick] (z1)--(z7);
	\draw[thick] (z1)--(z8);
	\draw[thick] (z2)--(z8);
	\draw[thick] (z2)--(z6);
	\draw[thick] (z2)--(z10);
	\draw[thick] (z3)--(z6);
	\draw[thick] (z3)--(z7);
	\draw[thick] (z3)--(z10);
	\draw[thick] (z3)--(z9);
	\draw[thick] (z4)--(z7);
	\draw[thick] (z4)--(z8);
	\draw[thick] (z5)--(z8);
	\draw[thick] (z5)--(z7);

	\coordinate (w1) at (-1,2);
	\coordinate (w2) at (0.25,1); 
	\coordinate (w3) at (-2.25,1); 
	\coordinate (w4) at (-1,-2);
	\coordinate (w5) at (-1,0);
	\coordinate (w6) at (-1, -3);

	\foreach \punto in {1,...,5}
	\fill (w\punto) circle(3pt);
	
	\node [above] at (w1) {$1$};
	\node [above] at (w2) {$2$};
	\node [above] at (w3) {$3$};
	\node [below] at (w4) {$4$};
	\node [below] at (w5) {$5$};
	
	\node [below] at (w6) {$F$};
	
	\draw[thick] (w1)--(w2);
	\draw[thick] (w3)--(w2);
	\draw[thick] (w1)--(w3);
	\draw[thick] (w2)--(w4);
	\draw[thick] (w3)--(w4);
	\draw[thick] (w3)--(w5);
	\draw[thick] (w2)--(w5);

	\end{tikzpicture}
	\caption{The graphs $G_F$ and $H_F$ are \textbf{NCSC}.}\label{fig: SCNC uno conexo y el otro disconexo}
\end{figure}
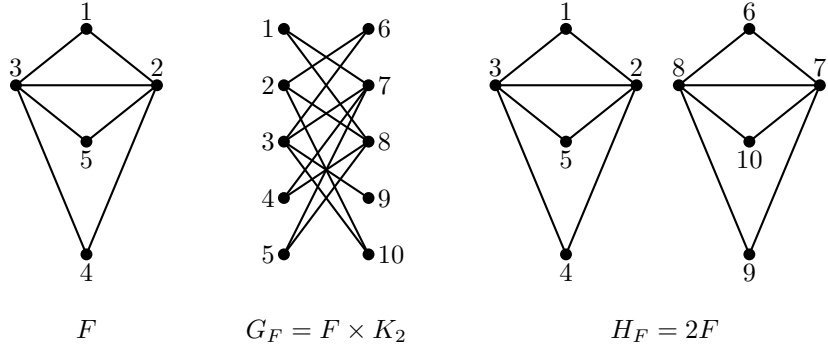

\begin{theorem}\label{GFHF}
	Let $F$ be a  nonbipartite graph with $|V(F)|=n$ and 
	$n\geq 3$. Then, the bipartite graph  $G_F=F\times K_2$  and  $H_F=2F$ (the union disjoint of two copies of $F$) are \textbf{NCSC}.
\end{theorem}
\begin{proof}
	Assume that $\sigma(F)=\{[\lambda_1]^{m_1},\cdots, [\lambda_{k}]^{m_{k}}\}$  with spectral decomposition $A_F=\sum_{i=1}^k \lambda_i P_i$. Notice that $\sigma(F)$ is not symmetric respect to zero, because $F$ is a nonbipartite graph (see Proposition~\ref{propiedades_basicas}).  Then, we obtain that

	\begin{eqnarray}
	A_{H_F}&=&\left(
	\begin{matrix}
	A_F& 0 \\
	\\
	0 & A_F
	\end{matrix}
	\right)=\sum_{i=1}^k\lambda_i \underbrace{\left(
		\begin{matrix}
		P_i& 0 \\
		\\
		0 & P_i
		\end{matrix}
		\right)}_{\hat{P}_i}=\sum_{i=1}^k \lambda_i \hat{P}_i,
	\end{eqnarray}
	and 
	\begin{eqnarray}
	A_{G_F}&=&\left(
	\begin{matrix}
	0& A_F \\
	\\
	A_F & 0
	\end{matrix}
	\right)=\sum_{i=1}^k\lambda_i \left(
	\begin{matrix}
	0& P_i \\
	\\
	P_i & 0
	\end{matrix}
	\right)\nonumber \\&=&\sum_{i=1}^k \lambda_i \left[\frac12\left(
	\begin{matrix}
	P_i& P_i \\
	\\
	P_i & P_i
	\end{matrix}
	\right) -\frac12 \left(
	\begin{matrix}
	P_i& -P_i \\
	\\
	-P_i & P_i
	\end{matrix}
	\right) \right] \nonumber\\&=&
	\sum_{i=1}^k \lambda_i \underbrace{\frac12\left(
		\begin{matrix}
		P_i& P_i \\
		\\
		P_i & P_i
		\end{matrix}
		\right)}_{\widetilde{P}_i^+} -\sum_{i=1}^k \lambda_i \underbrace{\frac12 \left(
		\begin{matrix}
		P_i& -P_i \\
		\\
		-P_i & P_i
		\end{matrix}
		\right)}_{\widetilde{P}_i^-}\nonumber\\&=&\sum_{i=1}^k \lambda_i \widetilde{P}_i^+ + \sum_{i=1}^k (-\lambda_i) \widetilde{P}_i^-.
	\end{eqnarray}
	Then, we conclude that 
	$$\sigma(G_F)=\{[-\lambda_1]^{m_1},\cdots, [-\lambda_{k}]^{m_{k}}, [\lambda_1]^{m_1},\cdots, [\lambda_{k}]^{m_{k}}\}$$
	and  $\sigma(H_F)=\{[\lambda_1]^{2m_1},\cdots, [\lambda_{k}]^{2m_{k}}\}$. This completes the proof.
\end{proof}

\begin{remark}\label{scnocoe}
	If $F=K_n$ with $n\ge 3$, then the $n-1$-regular bipartite graph $G_F=F\times K_2$ and the disconnected $n-1$-regular graph $2K_n$ are \textbf{NCSC}.
\end{remark}

In \cite{Row}, Rowlinson obtained
the relationship between the characteristic polynomials of a graph $G$ and the graph $G^*$ constructed by adding a new vertex to $G$, from its spectral decomposition. More precisely, the  graph modified by the addition of a vertex with any prescribed set of neighbours on $V(G)$. In order to prove our  result we need recall such statement.

\begin{theorem} [\cite{Row}, Theorem 2.1]\label{addvertex}
	Let $G$ be a finite graph whose adjacency matrix $A$ has spectral decomposition $A=\sum_{i=1}^m \mu_i P_i$. Let $\emptyset \neq S \subseteq V(G)=\{1, 2, \cdots, n\}$ and let $G^*$ be the
	graph obtained from G by adding one new vertex whose neighbours are the vertices
	in S. Then
	\begin{equation}
	P_{G^*}(x)=P_{G}(x)\left( x- \sum_{i=1}^m \frac{\rho_i^2}{x-\mu_i} \right),
	\end{equation}
	where $\rho_i=\|\sum_{k\in S} P_i e_k\|$ and $\{e_1,\cdots, e_n\}$ is the standard orthonormal basis of $\mathbb{R}^n$.
\end{theorem}

Combining theorems~\ref{GFHF} and~\ref{addvertex}, we obtain a construction that leads to a pair of~\textbf{NCSC} graphs by properly adding a vertex to the original pair of graphs. 

We will use the following notation for the below lemma. Let $F, G_F$ and $H_F$ as in Theorem \ref{GFHF}, where $V(F)=\{x_1,\ldots,x_n\}$ and $V(G_F)=V(H_F)=\{x_1,\ldots,x_n\}\cup \{y_1,\ldots,y_n\}$, where both sets correspond to the same copies of the vertices of $F$.  For any  $1\leq j\leq n,$ let $G_{F, j}^v$ be the graph obtained from $G_F$ by adding one new vertex $v$ whose neighbours are the vertices in $S=\{x_1,\cdots, x_j\}\cup \{y_1,\cdots, y_j\}$. Analogously, we define $H_{F, j}^w$ by adding a vertex $w$ and connecting it to the same set $S$ as $G_{F, j}^v$.

\begin{theorem}\label{polHF}
	There exist $\lambda_1,\ldots,\lambda_m$, $\rho_1,\ldots,\rho_m$, $m_1,\ldots,m_m$ such that
	\begin{equation*}
	P_{G_{F, j}^v}(x)=\prod_{i=1}^{m}(x-\lambda_i)^{2m_i-1}\left( x\prod_{i=1}^{m}(x-\lambda_i)- \sum_{i=1}^m \rho_i^2\prod_{j=1, j\neq i}^{m}(x-\lambda_j)\right)
	\end{equation*}
	and 
	\begin{equation*}
	P_{H_{F, j}^w}(x)=\prod_{i=1}^{m}(x-\lambda_i)^{m_i-1}\prod_{i=1}^{m}(x+\lambda_i)^{m_i}\left( x\prod_{i=1}^{m}(x-\lambda_i)- \sum_{i=1}^m \rho_i^2\prod_{j=1, j\neq i}^{m}(x-\lambda_j) \right).
	\end{equation*}
	In particular, $G_{F,j}^v$ and $H_{F, j}^w$ are \textbf{NCSC}.
\end{theorem}
\begin{proof}
	First of all, we label the vertices of both graphs as follows: $V(G_F)=\{z_1,\cdots, z_{2n}\}$ with $z_i=x_i$ and $z_{n+i}=y_i$, for $i\in \{1, \cdots, n\}$ and $V(H_F)=\{w_1,\cdots, w_{2n}\}$ where the first $n$ vertices belong to $F$. Then we denote by $V_j=\{1,\cdots,j, n+1,\cdots, n+j\},$ $S_{j}^{G_F}=\{z_k\}_{k\in V_j}$ and $S_{j}^{H_F}=\{w_k\}_{k\in V_j}$ (e.g. Fig.~\ref{fig: added vertex}). By Theorem \ref{GFHF} and \ref{addvertex}, we have that 
	\begin{equation}
	P_{G_{F, j}^v}(x)=P_{G_F}(x)\left( x- \sum_{i=1}^m \frac{\rho_i^2}{x-\lambda_i} \right),
	\end{equation}
	and 
	\begin{equation}
	P_{H_{F, j}^w}(x)=P_{H_F}(x)\left( x- \sum_{i=1}^m \frac{\sigma_i^2}{x-\lambda_i} - \sum_{i=1}^m \frac{\tau_i^2}{x+\lambda_i}\right),
	\end{equation}
	where $\rho_i=\|\sum_{k\in S_{G_F,j}} \hat{P}_i e_k\|,  \sigma_i=\|\sum_{k\in S_{H_F,j}} \widetilde{P}_i^+ e_k\|$ and $\tau_i=\|\sum_{k\in S_{H_F,j}} \widetilde{P}_i^- e_k\|.$ It follows that $\rho_i=\sigma_i$ and $\tau_i=0$ for any $1\leq i\leq m$. Since $P_{G_F}(x)=\prod_{i=1}^{m}(x-\lambda_i)^{2m_i}$ and $P_{H_F}(x)=\prod_{i=1}^{m}(x-\lambda_i)^{m_i}\prod_{i=1}^{m}(x+\lambda_i)^{m_i}$ we conclude that 
	\begin{equation*}
	P_{G_{F, j}^v}(x)=\prod_{i=1}^{m}(x-\lambda_i)^{2m_i-1}\left( x\prod_{i=1}^{m}(x-\lambda_i)- \sum_{i=1}^m \rho_i^2\prod_{j=1, j\neq i}^{m}(x-\lambda_j)\right)
	\end{equation*}
	and 
	\begin{equation*}
	P_{H_{F, j}^w}(x)=\prod_{i=1}^{m}(x-\lambda_i)^{m_i-1}\prod_{i=1}^{m}(x+\lambda_i)^{m_i}\left( x\prod_{i=1}^{m}(x-\lambda_i)- \sum_{i=1}^m \rho_i^2\prod_{j=1, j\neq i}^{m}(x-\lambda_j) \right).
	\end{equation*}
	
\end{proof}

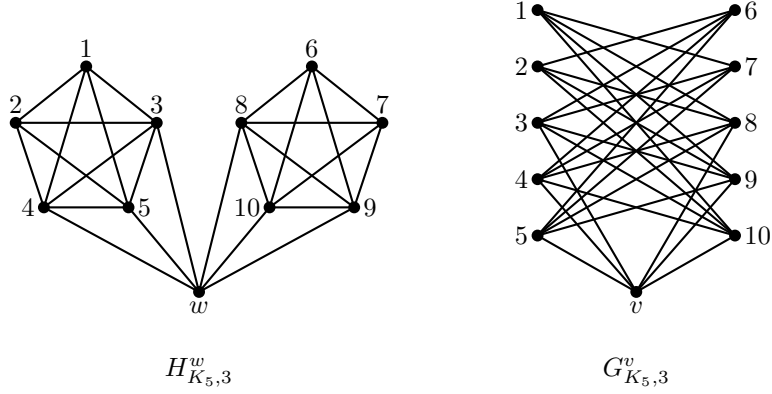
\begin{figure}
	\centering
	
	\begin{tikzpicture}[scale=0.75]
	
	\coordinate (w5) at (0.75,-1.5);
	\coordinate (w4) at (-0.75,-1.5);
	\coordinate (w2) at (-1.25,0); 
	\coordinate (w3) at (1.25,0); 
	\coordinate (w1) at (0,1);
	\coordinate (w6) at (4,1);
	\coordinate (w7) at (5.25,0); 
	\coordinate (w8) at (2.75,0); 
	\coordinate (w9) at (4.75,-1.5); 
	\coordinate (w10) at (3.25,-1.5);
	\coordinate (w11) at (2, -3);
	\coordinate (w12) at (2, -4);

	\foreach \punto in {1,...,11}
	\fill (w\punto) circle(3pt);
	
	\node [above] at (w1) {$1$};
	\node [above] at (w3) {$3$};
	\node [above] at (w2) {$2$};
	\node [left] at (w4) {$4$};
	\node [right] at (w5) {$5$};
	\node [above] at (w6) {$6$};
	\node [above] at (w7) {$7$};
	\node [above] at (w8) {$8$};
	\node [right] at (w9) {$9$};
	\node [left] at (w10) {$10$};
	\node [below] at (w11) {$w$};
	
	\node [below] at (w12) {$H_{K_5,3}^w$};
	
	\foreach \j in {1,...,5}
	\foreach \k in {\j,...,5}
	\draw[thick] (w\j)--(w\k);
	
	\foreach \j in {6,...,10}
	\foreach \k in {\j,...,10}
	\draw[thick] (w\j)--(w\k);

	\draw[thick] (w11)--(w3);
	\draw[thick] (w11)--(w5);
	\draw[thick] (w11)--(w4);
	\draw[thick] (w11)--(w10);
	\draw[thick] (w11)--(w8);
	\draw[thick] (w11)--(w9);
	
	\coordinate (z1) at (8,2);
	\coordinate (z2) at (8,1); 
	\coordinate (z3) at (8,0); 
	\coordinate (z4) at (8,-1);
	\coordinate (z5) at (8,-2);
	\coordinate (z6) at (11.5, 2);
	\coordinate (z7) at (11.5,1);
	\coordinate (z8) at (11.5, 0); 
	\coordinate (z9) at (11.5,-1); 
	\coordinate (z10) at (11.5,-2);
	\coordinate (z11) at (9.75,-3);
	\coordinate (z12) at (9.75, -4);
	
	\foreach \punto in {1,...,11}
	\fill (z\punto) circle(3pt);
	
	\node [left] at (z1) {$1$};
	\node [left] at (z2) {$2$};
	\node [left] at (z3) {$3$};
	\node [left] at (z4) {$4$};
	\node [left] at (z5) {$5$};
	\node [right] at (z6) {$6$};
	\node [right] at (z7) {$7$};
	\node [right] at (z8) {$8$};
	\node [right] at (z9) {$9$};
	\node [right] at (z10) {$10$};
	\node[below] at (z11) {$v$};
	
	\node[below] at (z12) {$G_{K_5,3}^v$};

	\draw[thick] (z1)--(z7);
	\draw[thick] (z1)--(z8);
	\draw[thick] (z1)--(z9);
	\draw[thick] (z1)--(z10);
	\draw[thick] (z2)--(z8);
	\draw[thick] (z2)--(z9);
	\draw[thick] (z2)--(z10);
	\draw[thick] (z2)--(z6);
	\draw[thick] (z3)--(z6);
	\draw[thick] (z3)--(z7);
	\draw[thick] (z3)--(z10);
	\draw[thick] (z3)--(z9);
	\draw[thick] (z4)--(z10);
	\draw[thick] (z4)--(z8);
	\draw[thick] (z4)--(z6);
	\draw[thick] (z4)--(z7);
	\draw[thick] (z5)--(z6);
	\draw[thick] (z5)--(z7);
	\draw[thick] (z5)--(z8);
	\draw[thick] (z5)--(z9);
	\draw[thick] (z11)--(z3);
	\draw[thick] (z11)--(z4);
	\draw[thick] (z11)--(z5);
	\draw[thick] (z11)--(z10);
	\draw[thick] (z11)--(z9);
	\draw[thick] (z11)--(z8);
	
	\end{tikzpicture}
	\caption{In this example $F=K_5$, the vertices of the graphs $G_F$ and $H_F$ are labeled from $1$ to $10$ consecutive in each connected component of $G_F$ and each set of the bipartition of $H_F$ respectively, and $S=\{3,4,5,8,9,10\}$ is the neighbourhood of the new added vertex. Both graphs are \textbf{NSCS} graphs.}\label{fig: added vertex}
\end{figure}

\begin{corollary}\label{gnj}
	It holds that 
	$$
	\sigma(G_{K_n, j}^v)=\{[-(n-1)]^{1}, [-1]^{n-2}, [1]^{n-1}, \lambda_1, \lambda_2, \lambda_3\}
	$$
	and 
	$$
	\sigma(H_{K_n, j}^w)=\{[-1]^{2n-3}, [n-1]^{1}, \lambda_1, \lambda_2, \lambda_3\}, 
	$$	
	with $\lambda_1 >0, \lambda_3<0$ and 
	$$
	\lambda_2= \left\{ \begin{array}{lcc}
	>0 &   {\rm if}   & 1\leq j\leq n-2 \\
	0 &  {\rm if} & j=n-1 \\
	<0 &  {\rm if}  & j=n
	\end{array}
	\right.
	$$
	In particular, $G_{K_n,j}^v$ and $H_{K_n, j}^w$ are \textbf{NCSC}.
\end{corollary}
\begin{proof}
	

	Let us suppose that the eigenvalues of $G_{K_n}$ and $H_{K_n}$ are denoted by $\{\mu_1,\mu_2, \mu_3, \mu_4\}$ and $\{\nu_1, \mu_2\}$, respectively,   in increasing order and 
	\begin{equation}
	A_{G_{K_n}}=\sum_{r=1}^4 \mu_r P_r^{G_{K_n}}\qquad {\rm and} \qquad A_{H_{K_n}}=\sum_{s=1}^2 \nu_s P_s^{H_{K_n}},
	\end{equation}
	where $P_r^{G_{K_n}}$ and $P_s^{H_{K_n}}$ represent the orthogonal projection onto the eigenspace $E_{\mu_r}$  and $E_{\nu_s}$, respectively. 
	
	
	For sake of simplicity, we  denote by $G_{K_n, j}$ and $H_{K_n, j}$ the graphs omitting in the notation the added vertex in each one of them. From Theorem  \ref{polHF}, for each $j\in \{1, 2, \cdots, n\}$ we obtain the characteristic polynomials of $G_{K_n,j}$ and $H_{K_n,j}$, respectively:

	\begin{equation}\label{polcarmod}
	\begin{split}
	P_{G_{K_n,j}}(x)=P_{G_{K_n}}(x)\left( x- \sum_{r=1}^4 \frac{\rho_r^2}{x-\mu_r} \right), \\  	P_{H_{K_n,j}}(x)=P_{H_{K_n}}(x)\left( x- \sum_{s=1}^2 \frac{\sigma_s^2}{x-\nu_s} \right), 
	\end{split}
	\end{equation}
	where $\rho_r=\|\sum_{k\in S_j^{G_{K_n}}}P_r^{G_{K_n}}e_k\|$,  $\sigma_s=\|\sum_{k\in S_j^{H_{K_n}}}P_s^{H_{K_n}}e_k\|$ and $S_j^{G_{K_n}}=S_j^{H_{K_n}}=\{1,\cdots, j, n+1, \cdots, n+j\}.$
	
	For the sake of clarity, given a matrix $B\in \mathbb{R}^{n \times n}$ and $k\in \{1, \cdots, n\}$ we denote by  $c_k(B)\in \mathbb{R}^n$ the  $k^{\rm th}$ column of  $B$. Then, 
	\begin{enumerate}
		\item $\rho_1=\|\sum_{k\in S_j^{G_{K_n}}}P_1^{G_{K_n}}e_k\|=\|\sum_{k\in S_j^{G_{K_n}}}c_k(P_1^{G_{K_n}})\|=0.$
		\item $\rho_2^2=\|\sum_{k\in S_j^{G_{K_n}}}P_2^{G_{K_n}}e_k\|^2=\|\sum_{k\in S_j^{G_{K_n}}}c_k(P_2^{G_{K_n}})\|^2=\frac{2j(n-j)}{n}.$
		\item $\rho_3=\|\sum_{k\in S_j^{G_{K_n}}}P_3^{G_{K_n}}e_k\|=\|\sum_{k\in S_j^{G_{K_n}}}c_k(P_3^{G_{K_n}})\|=0.$
		\item $\rho_4=\|\sum_{k\in S_j^{G_{K_n}}}P_4^{G_{K_n}}e_k\|=\|\sum_{k\in S_j^{G_{K_n}}}c_k(P_4^{G_{K_n}})\|=\frac{j}{n}\| c_1(J_{2n})\|=\frac{j}{n}(2n)^{1/2}.$
		\item $\sigma_1^2=\|\sum_{k\in S_j^{H_{K_n}}}P_1^{H_{K_n}}e_k\|^2=\|\sum_{k\in S_j^{H_{K_n}}}c_k(P_1^{H_{K_n}})\|^2=\frac{2j(n-j)}{n}.$
		\item $\sigma_2=\|\sum_{k\in S_j^{H_{K_n}}}P_2^{H_{K_n}}e_k\|=\|\sum_{k\in S_j^{H_{K_n}}}c_k(P_2^{H_{K_n}})\|=\frac{j}{n}\| c_1(J_{2n})\|=\frac{j}{n}(2n)^{1/2}.$
	\end{enumerate}
	

	Now we have to consider three possible cases.
	\begin{enumerate}
		\item {\bf Case 1:} $j=n.$
		
		Then from \eqref{polcarmod} we can assert that:
		\begin{eqnarray}
		P_{G_{K_n,j}}(x)&=&P_{G_{K_n}}(x)\left( x- \frac{2n}{x-(n-1)} \right)\nonumber\\
		&=&\frac{P_{G_{K_n}}(x)}{(x-(n-1))}(x^2-(n-1)x-2n),
		\end{eqnarray}
		and 
		\begin{eqnarray}
		P_{H_{K_n,j}}(x)=\frac{P_{H_{K_n}}(x)}{(x-(n-1))}(x^2-(n-1)x-2n).
		\end{eqnarray}
		Let $Q_1(x)=(x^2-(n-1)x-2n)$ and $\lambda_1, \lambda_3$ its roots, it is easily seen that $\lambda_1\lambda_3<0$. Finally,  letting $\lambda_2=-1<0$ 	we conclude this part of the proof. 
		\item {\bf Case 2:} $j=n-1.$
		
		In this case from \eqref{polcarmod} we obtain that the characteristic polynomials have the following expressions: 
		
		\begin{eqnarray}
		P_{G_{K_n,j}}(x)&=&\frac{P_{G_{K_n}}(x)}{(x-(n-1))(x+1)}x\left(nx^2+n(2-n) x- (3n^2-7n+4) \right)\nonumber\,
		\end{eqnarray}
		and 
		\begin{eqnarray}
		P_{H_{K_n,j}}(x)&=&\frac{P_{H_{K_n}}(x)}{(x-(n-1))(x+1)}x\left(nx^2+n(2-n) x- (3n^2-7n+4) \right).\nonumber\
		\end{eqnarray}
		Let $Q_2(x)=x\left(nx^2+n(2-n) x- (3n^2-7n+4) \right)$ and $\lambda_1, \lambda_2, \lambda_3$ its roots with $\lambda_2=0$.  It is easily seen that $\lambda_1\lambda_3<0$ and this concludes the proof in  this case. 
		\item {\bf Case 3:} $1\leq j \leq n-2.$
		
		We have that 
		\begin{eqnarray}
		P_{G_{K_n,j}}(x)&=&\frac{P_{G_{K_n}}(x)}{(x-(n-1))(x+1)}Q_3(x),\nonumber\
		\end{eqnarray}
		and 
		\begin{eqnarray}
		P_{H_{K_n,j}}(x)&=&\frac{P_{H_{K_n}}(x)}{(x-(n-1))(x+1)}Q_3(x), \nonumber\
		\end{eqnarray}
		where	$Q_3(x)=x^3+(2-n) x^2- (n-1+2j)x+2j((n-1)-j).$ 
		
		We denote by $\lambda_1, \lambda_2, \lambda_3$ the real roots of $Q_3$ such that $\lambda_3\leq \lambda_2\leq \lambda_1.$ We know that $\lambda_3<0$, since $\lim\limits_{x\to -\infty} Q_3(x)=-\infty$ and $Q_3(0)=2j((n-1)-j)>0.$ On the other hand, by  Descartes' rule of signs, we conclude that $\lambda_2$ and $\lambda_1$ are positive real numbers and  this finishes the proof.

	\end{enumerate}
\end{proof}
\begin{remark}
	From the previous proof we obtain that 
	\begin{equation}\label{polcargnj}
	P_{G_{K_n,j}}(x)=(x-1)^{n-1}(x+1)^{n-2}(x+(n-1))Q_{n, j}(x)
	\end{equation}
	and 
	\begin{equation}\label{polcarhnj}
	P_{H_{K_n,j}}(x)=(x-1)^{2n-3}(x-(n-1))Q_{n, j}(x), 
	\end{equation}
	with $Q_{n, j}(x)=x(x+1)(x-(n-1))-2j(x-(n-1-j)).$
\end{remark}

Next, we present a generalization of the previous construction via coalescence operation between two graphs. If  $G$ and $H$ are two graphs,  $g\in V(G)$ and $h\in V(H)$,  the \emph{coalescence} between $G$ and $H$ at $g$ and $h$, denoted   $G\cdot H (g,h:v_{g,h})$ or $G\cdot_{g}^h H$, is the graph obtained from $G$ and $H$, by identifying vertices $g$ and $h$ (see Fig.~\ref{fig: coalescence}).  We use $G\cdot H$ for short.

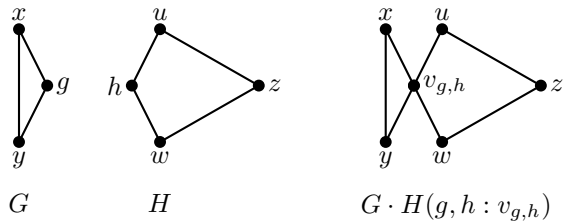
\begin{figure}
	\centering
	\begin{tikzpicture}[scale=0.75]
	
	\coordinate (w5) at (-0.75,0);
	\coordinate (w4) at (0.75,0);
	\coordinate (w2) at (1.25,1); 
	\coordinate (w3) at (-1.25,1); 
	\coordinate (w6) at (1.25,-1);
	\coordinate (w7) at (-1.25,-1);
	\coordinate (w8) at (3,0);
	\coordinate (w9) at (-1.25,-1.75);
	\coordinate (w10) at (1.25,-1.75);
	
	\foreach \punto in {2,...,8}
	\fill (w\punto) circle(3pt);
	
	\node [above] at (w3) {$x$};
	\node [above] at (w2) {$u$};
	\node [left] at (w4) {$h$};
	\node [right] at (w5) {$g$};
	\node [below] at (w6) {$w$};
	\node [below] at (w7) {$y$};
	\node [right] at (w8) {$z$};
	
	\node [below] at (w9) {$G$};
	\node [below] at (w10) {$H$};
	%
	\draw[thick] (w3)--(w7);
	\draw[thick] (w3)--(w5);
	\draw[thick] (w5)--(w7);
	\draw[thick] (w2)--(w4);
	\draw[thick] (w4)--(w6);
	\draw[thick] (w6)--(w8);
	\draw[thick] (w2)--(w8);
	%
	%
	
	\coordinate (y5) at (5.75,0);
	\coordinate (y4) at (5.75,0);
	\coordinate (y2) at (6.25,1); 
	\coordinate (y3) at (5.25,1); 
	
	\coordinate (y6) at (6.25,-1);
	\coordinate (y7) at (5.25,-1);
	\coordinate (y8) at (8,0);
	\coordinate (y9) at (6.5,-1.75);

	\foreach \punto in {2,...,8}
	\fill (y\punto) circle(3pt);

	\node [above] at (y3) {$x$};
	\node [above] at (y2) {$u$};
	\node [right] at (y4) {$v_{g, h}$};
	\node [below] at (y6) {$w$};
	\node [below] at (y7) {$y$};
	\node [right] at (y8) {$z$};
	
	\node [below] at (y9) {$G\cdot H (g,h:v_{g,h})$};

	\draw[thick] (y7)--(y3);
	\draw[thick] (y3)--(y4);
	\draw[thick] (y4)--(y7);
	\draw[thick] (y2)--(y4);
	\draw[thick] (y2)--(y8);
	\draw[thick] (y6)--(y8);
	\draw[thick] (y4)--(y6);

	\end{tikzpicture}
	\caption{The coalescence of graphs $G$ and $H$ at vertices $g$ and $h$.}
	\label{fig: coalescence}
	
\end{figure}

In the 70s Schwenk published an article containing useful formulas for the characteristic polynomial of a graph~\cite{Schwenk1974}. The  main result of this research is based on the following Schwenk's formula, linking the characteristic polynomial of two graphs and the coalescence between them. More details on next result can be found in~\cite{CRS-1997}.

\begin{lemma}\cite{Schwenk1974}\label{lem: Schwenk}
	Let $G$ and $H$ be two graphs. If $g\in V(G)$, $h\in V(H)$, and $F=G\cdot H$, then
	\[P_F(x)=P_G(x)P_{H-h}(x)+P_{G-g}(x)P_H(x)-xP_{G-g}(x)P_{H-h}(x).\]  
\end{lemma}

\begin{remark}
	Combining Proposition \ref{scnocoe} and Theorem \ref{addvertex}, we obtain a pair of graphs singulary cospectral but not (almost) cospectral via coalescence. More precisely, 	let  $G_{K_n, j}^v, H_{K_n, j}^w$ with $n\geq 3$ and  $1\leq j\leq n$   as above.  For each $k\geq 1$, the graph  $G_{K_n, j, k}^v$  is  defined inductively as follows:  $ G_{K_n, j, 1}^{v_1}:=G_{K_n, j}^{v}$  and $G_{K_n, j, k}^{v_{k}}= G_{K_n, j, k-1}^{v_{k-1}}\cdot G_{K_n, j}^{\widetilde{v}_{k-1}}(v_{k-1}, \widetilde{v}_{k-1}:v_k)$ for $k\geq 2.$ Analogously, we consider the graphs  $H_{K_n, j, 1}^{w_1}:=H_{K_n, j}^{w}$ and  $H_{K_n, j, k}^{w_k}= H_{K_n, j, k-1}^{w_{k-1}}\cdot H_{K_n, j}^{\widetilde{w}_{k-1}} (w_{k-1}, \widetilde{w}_{k-1}:w_k).$  Then $G_{K_n,j, k}^{v_k}$ and $H_{K_n, j, k}^{w_k}$ are \textbf{NCSC}.
\end{remark}

In Remark~\ref{scnocoe} we present families of pairs of regular graphs which are singularly cospectral. To finish the section we will show that if we have a pair of singularly cospectral graphs, on the same number of vertices, and one of them is regular, then the other one is also regular.   

\begin{proposition}\label{igual_max_autovalor}
	Let $G, H$ be  two graphs   such that $G$ is a $d$-regular graph,   $\lambda_1(G)=\lambda_1(H)$ and $d=\frac{2|E(H)|}{|V(H)|}$ then
	$H$ is a $d$-regular.

\end{proposition}
\begin{proof}
	We denote by $n=|V(G)|$ and $m=|E(G)|.$ 	As $G$ is a regular graph, by Proposition \ref{propiedades_basicas} we have that  $$d=\frac{2m}{n}=\lambda_1(G)=\lambda_1(H).$$ 
	Then,  by Proposition \ref{propiedades_basicas} we can conclude that $H$ is a $d$-regular graph.
\end{proof}

\begin{corollary}
	Let $G, H$ be two graphs with the same numbers of vertices, such that $G$ is a $d$-regular graph and $H$ a non-regular one. Then $G$  cannot  be singularly cospectral with $H$.
\end{corollary}
\begin{proof}
	As in the previous stamenent we denote by $n=|V(G)|=|V(H)|$ and $m=|E(G)|.$ Suppose that $G$ and $H$ are singularly cospectral, then from Proposition \ref{prop_nec} and \ref{propiedades_basicas} we have $m=|E(H)|$ and 
	$$d=\frac{2|E(H)|}{|V(H)|}=\lambda_1(G)=s_1(G)=s_1(H)=\lambda_1(H),$$
	thus $H$ is a $d$-regular graph as consequence of Proposition \ref{igual_max_autovalor}. But this contradict our hypothesis and conclude the proof.
\end{proof}

\section{Families of graphs where singularly cospectral implies almost cospectral}~\label{sec: c implies sc}

Since for two graphs to be singularly coespectral is not necessary they to have the same number of vertices, the notion of almost cospectral is nearer to singularly cospectral than cospectral. Two graphs are {\em almost cospectral} if their nonzero eigenvalues (and their multiplicities) coincide. The connected components of $P_3\times P_3$, $C_4$ and $K_{1,4}$ are almost cospectral, see \cite[Theorem 3.16]{CvetkovicDDGT88}. For more details and results in connection with almost cospectral graphs we referred to the reader to~\cite{BeinekeLW04} and the references therein. Notice that if two almost cospectral graphs having the same number of vertices, then they are cospectral. Hence, in Section~\ref{sec: construction sc noncospectral}, we have presented constructions of singularly cospectral graphs which are not almost cospectral. This section is devoted to present families of graphs where the notion of singularly cospectral and cospectral are equivalent, namely, bipartite graphs (Theorem~\ref{scbipartito}), connected graphs having maximum singular value with multiplicity at least two (Theorem~\ref{mayoroiguala2}), and connected graphs having the same inertia (Theorem~\ref{thm: inertia}).

\begin{theorem}\label{scbipartito}
	Let $G, H$ be two bipartite graphs and singularly cospectral then they are almost cospectral.
\end{theorem}
\begin{proof}
	Let $\lambda_1\geq \lambda_2\geq \cdots\geq\lambda_n$ and $\mu_1\geq \mu_2\geq \cdots\geq \mu_m$ be the non-zero eigenvalues of $G$ and $H$,
	respectively.  As $G$ and $H$ are singularly cospectral then $n=m$ and  the spectrum of a bipartite graph is symmetric about $0$  then $n$ is even. It follows that, for any $1\leq i\leq \frac{n}{2}$, it holds  
	$$\lambda_i=-\lambda_{n-i}\qquad {\rm and} \qquad 	\mu_i=-\mu_{n-i}.$$ Since the singular values of a symmetric matrix are the
	absolute values of its nonzero eigenvalues, then 
	$$\lambda_1, \lambda_1, \lambda_2, \lambda_2, \cdots, \lambda_{\frac{n}{2}}, \lambda_{\frac{n}{2}}
	$$
	and 
	$$\mu_1, \mu_1, \mu_2, \mu_2, \cdots, \mu_{\frac{n}{2}}, \mu_{\frac{n}{2}},
	$$
	are the singular values of $G$ and $H$, respectively. This shows that the graphs are almost cospectral.
\end{proof}

\begin{theorem}\label{mayoroiguala2}
	Let $G, H$ be two connected  and singularly cospectral graphs such that its largest singular value has multiplicity greater or  equal to  2. Then, $G$ and $H$ are almost cospectral 
\end{theorem}

\begin{proof}
	Let $s_1> s_2>\cdots> s_k$  be the non-zero  singular values of $G$ with multiplicity $m_1,m_2,\cdots, m_k$, respectively. 	If we denote by $\lambda_1(G)$ and $\mu_1(H)$ the largest eigenvalue of $G$ and $H$, respectively, then $s_1=\lambda_1(G)=\mu_1(H)$. On the other hand, by the connectivity of both graphs such eigenvalues are simple. As by hypothesis, $s_1$ has multiplicity greater or  equal to  2, then  the smallest eigenvalue of $G$ and $H$ is equal to $-\lambda_1(G).$ So by Proposition \ref{propiedades_basicas} and Theorem \ref{scbipartito}, we conclude that both graphs are almost cospectral.
	
\end{proof}

\begin{remark}
	The connectivity condition can not be dropped as hypothesis from~Theorem~\ref{mayoroiguala2}. See Fig.~\ref{fig: connectivity theorem 5.2} for an example of two graphs, one of them disconnected which are singularly cospectral but not almost cospectral.
\end{remark}	


\begin{figure}
	\centering
	\begin{tikzpicture}[scale=0.75]
	\coordinate (n1) at (-4,0);
	\coordinate (n2) at (-3,1); 
	\coordinate (n3) at (-4,2);
	\coordinate (n4) at (-5,1);
	\coordinate (n5) at (-2,1);
	\coordinate (n6) at (-1,0);
	\coordinate (n7) at (0,1);
	\coordinate (n8) at (-1,2);
	\coordinate (n9) at (-2.5, -1);
	
	\foreach \punto in {1,...,8}
	\fill (n\punto) circle(3pt);
	\node [below] at (n1) {$1$};
	\node [above] at (n3) {$3$};
	\node [above] at (n2) {$2$};
	\node [above] at (n4) {$4$};
	\node [above] at (n5) {$5$};
	\node [below] at (n6) {$6$};
	\node [above] at (n7) {$7$};
	\node [above] at (n8) {$8$};
	\node [] at (n9) {$H_1$};
	
	\draw[thick] (n1)--(n2);
	\draw[thick] (n1)--(n3);
	\draw[thick] (n2)--(n3);
	\draw[thick] (n2)--(n4);
	\draw[thick] (n3)--(n4);
	\draw[thick] (n4)--(n1);
	
	\draw[thick] (n5)--(n6);
	\draw[thick] (n5)--(n7);
	\draw[thick] (n6)--(n7);
	\draw[thick] (n6)--(n8);
	\draw[thick] (n7)--(n8);
	\draw[thick] (n8)--(n5);
	
	\coordinate (y1) at (3,0);
	\coordinate (y2) at (5,0); 
	\coordinate (y3) at (5,2);
	\coordinate (y4) at (3,2);
	\coordinate (y5) at (3.75,0.75);
	\coordinate (y6) at (5.75,0.75);
	\coordinate (y7) at (5.75,2.75);
	\coordinate (y8) at (3.75,2.75);
	\coordinate (y9) at (4, -1);
	
	\foreach \punto in {1,...,8}
	\fill (y\punto) circle(3pt);
	
	\node [below] at (y1) {$1$};
	\node [below] at (y2) {$2$};
	\node [above] at (y3) {$3$};
	\node [above] at (y4) {$4$};
	\node [below] at (y5) {$5$};
	\node [below] at (y6) {$6$};
	\node [above] at (y7) {$7$};
	\node [above] at (y8) {$8$};
	\node [] at (y9) {$H_2$};

	\draw[thick] (y1)--(y2);
	\draw[thick] (y1)--(y4);
	\draw[thick] (y3)--(y4);
	\draw[thick] (y3)--(y2);
	\draw[thick] (y4)--(y8);
	\draw[thick] (y6)--(y7);
	\draw[thick] (y7)--(y3);
	\draw[thick] (y2)--(y6);
	\draw[thick] (y1)--(y5);
	\draw[thick] (y5)--(y6);
	\draw[thick] (y5)--(y8);
	\draw[thick] (y7)--(y8);
	\end{tikzpicture}
	
	\caption{
		$
		\sigma(H_1)=\{[-1]^6, [3]^2\}
		$
		and 	$
		\sigma(H_2)=\{[-3]^{1}, [-1]^{3}, [1]^{3}, [3]^1\}
		$.}\label{fig: connectivity theorem 5.2}
\end{figure}
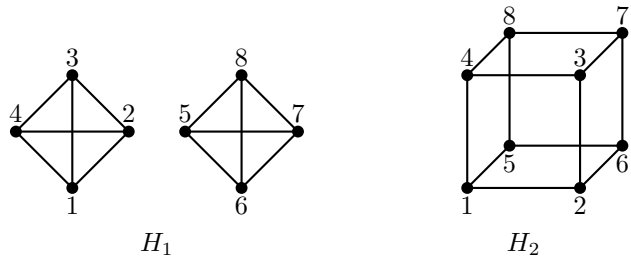


\begin{proposition}\label{inertia_cond}
	Let $G, H$ be two connected  and singularly cospectral graphs with  exactly three different singular values, without considering their multiplicities,  such that $n(G)=n(H)$ or $p(G)=p(H).$ Then they are almost cospectral.	
\end{proposition}

\begin{proof}
	Let $s_1> s_2> s_3$  be the non-zero  singular values of $G$ with multiplicity $m_1,m_2, m_3$, respectively. 	If we denote by $\lambda_1(G)$ and $\mu_1(H)$ the largest eigenvalue of $G$ and $H$, respectively, then $s_1=\lambda_1(G)=\mu_1(H)$. On the other hand, by the connectivity of both graphs such eigenvalues are simple.

	The non-zero eigenvalues of $G$ and $H$ are include in the set 
	$$
	\{\pm s_1, \pm s_2, \pm s_3\}.
	$$
	Even more, for any  $i\in \{1, 2, 3\}$,   the non-zero eigenvalues of $G$ are exactly $s_i$ and $-s_i$ with multiplicity  $m_i^+$ and $m_i^-$ respectively.  Analogously, the non-zero eigenvalues of  $H$ are $s_i$ and $-s_i$ with multiplicity    $\hat{m}_i^+$ and $\hat{m}_i^-$, respectively. Obviously, it holds that 
	$m_i^++m_i^-=m_i=\hat{m}_i^++\hat{m}_i^-$ with 
	$m_i^+, m_i^-, \hat{m}_i^+, \hat{m}_i^- \geq 0.$

	This gives two cases to consider:
	\begin{enumerate}
		\item {\bf Case 1:}  $m_1\geq 2$.\\
		From Proposition \ref{mayoroiguala2} we have that $G$ and $H$ are almost cospectral.
		
		\item {\bf Case 2:} $m_1=1.$

		As  $m_1=1$  thus $m_i^+=\hat{m}_i^+=1$ and  $m_i^-=\hat{m}_i^-=0.$ Using the well-known fact that the sum of all eigenvalues of a graph is always zero, we have the following equalities
		
		\begin{equation}\label{sistema_s}
		\left\{ \begin{array}{lcc}
		s_1+(m_2^+-m_2^-)s_2+(m_3^+-m_3^-)s_3=0. \\
		\\ 
		s_1+(\hat{m}_2^+-\hat{m}_2^-)s_2+(\hat{m}_3^+-\hat{m}_3^-)s_3=0.
		\end{array}
		\right.
		\end{equation}
		
		From the hypothesis about the positive or negative inertia one of the following identity holds
		\begin{equation}\label{positivo}
		1+m_2^++m_3^+=1+\hat{m}_2^++\hat{m}_3^+, 
		\end{equation}
		or 
		\begin{equation}\label{negativo}
		m_2^-+m_3^-=\hat{m}_2^-+\hat{m}_3^-. 
		\end{equation}
		
		From now on, without loss of generality, we assume that $p(G)=p(H)$. Then, by  \eqref{sistema_s} and  \eqref{positivo}, we have that 
		$$
		(m_2^+-\hat{m}_2^+)(s_2-s_3)=0.
		$$
		As $s_2-s_3>0$, then we conclude that $m_2^+-\hat{m}_2^+=0=m_3^+-\hat{m}_3^+$. This shows that $G$ and $H$ are almost cospectal graphs and it concludes the proof.
	\end{enumerate}

\end{proof}

\begin{remark}
	The condition about the negative or positive inertia cannot be dropped from Proposition~\ref{inertia_cond}, as can
	be seen in Fig.~\ref{fig: connectivity theorem 5.3}.
\end{remark}

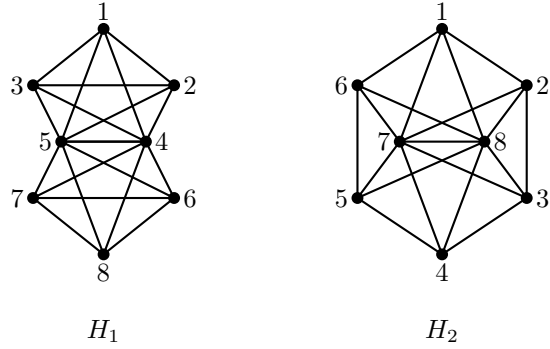
\begin{figure}
	\centering
	\begin{tikzpicture}[scale=0.75, every node/.style={scale=1}]
	
	\coordinate (w5) at (-0.75,0);
	\coordinate (w4) at (0.75,0);
	\coordinate (w2) at (1.25,1); 
	\coordinate (w3) at (-1.25,1); 
	\coordinate (w1) at (0,2);
	\coordinate (w6) at (1.25,-1);
	\coordinate (w7) at (-1.25,-1);
	\coordinate (w8) at (0,-2);

	\foreach \punto in {1,...,8}
	\fill (w\punto) circle(3pt);
	
	\node [above] at (w1) {$1$};
	\node [left] at (w3) {$3$};
	\node [right] at (w2) {$2$};
	\node [right] at (w4) {$4$};
	\node [left] at (w5) {$5$};
	\node [right] at (w6) {$6$};
	\node [left] at (w7) {$7$};
	\node [below] at (w8) {$8$};
	
	\node [below] at (0,-3) {$H_1$};
	
	\foreach \j in {1,...,5}
	\foreach \k in {\j,...,5}
	\draw[thick] (w\j)--(w\k);
	
	\foreach \j in {4,...,8}
	\foreach \k in {\j,...,8}
	\draw[thick] (w\j)--(w\k);
	
	\coordinate (y1) at (6,2);
	\coordinate (y2) at (7.50,1); 
	\coordinate (y3) at (7.50,-1);
	\coordinate (y4) at (6,-2);
	\coordinate (y5) at (4.50,-1);
	\coordinate (y6) at (4.50,1); 
	\coordinate (y7) at (5.25,0);
	\coordinate (y8) at (6.75,0);

	\foreach \punto in {1,...,8}
	\fill (y\punto) circle(3pt);
	
	\node [above] at (y1) {$1$};
	\node [right] at (y3) {$3$};
	\node [right] at (y2) {$2$};
	\node [below] at (y4) {$4$};
	\node [left] at (y5) {$5$};
	\node [left] at (y6) {$6$};
	\node [left] at (y7) {$7$};
	\node [right] at (y8) {$8$};
	\node [below] at (6,-3) {$H_2$};
	
	\draw[thick] (y1)--(y2);
	\draw[thick] (y2)--(y3);
	\draw[thick] (y3)--(y4);
	\draw[thick] (y4)--(y5);
	\draw[thick] (y5)--(y6);
	\draw[thick] (y6)--(y1);
	
	\draw[thick] (y7)--(y8);
	
	\foreach \j in {1,...,6}
	{
		\draw[thick] (y\j)--(y7);
		\draw[thick] (y\j)--(y8);
	}
	
	\end{tikzpicture}
	\caption{$
		\sigma(H_1)=\{[-2]^{1}, [-1]^{5}, [2]^{1}, [5]^1\}
		$
		and
		$
		\sigma(H_2)=\{[-2]^2, [-1]^3, [1]^2, [5]^1\}.
		$}\label{fig: connectivity theorem 5.3}
\end{figure}


As a consequence of Proposition~\ref{inertia_cond} the following result holds.
\begin{theorem}\label{thm: inertia}
	Let $G, H$ be two connected  and singularly cospectral graphs with  exactly three different singular values, without count its multiplicities,  such that $In(G)=In(H)$.  Then, $G$ and $H$ are almost cospectral.		
\end{theorem}

In the above Theorem almost cospectral can be replaced by cospectral because $|V(G)|=|V(H)|$.

\end{document}